\documentclass[12pt]{amsart}
\usepackage[english]{babel}
\usepackage[all,cmtip]{xy}
\usepackage{amsmath, mathrsfs, amsfonts}
\usepackage{color}

\newcommand{\wt}[1]{\widetilde{#1}} 
\newcommand{\ol}[1]{\overline{#1}} 
\newcommand{\wh}[1]{\widehat{#1}} 
\newcommand{\twogroup}{\mu_{2}}

\newtheorem{theorem}{Theorem}

\newtheorem{corollary}[theorem]{Corollary}
\newtheorem{lemma}[theorem]{Lemma}
\newtheorem{proposition}[theorem]{Proposition}
\newtheorem{remark}[theorem]{Remark}

\newtheorem{proposition-definition}[theorem]{Proposition-Definition}

\newtheorem{conjecture}[theorem]{Conjecture}
\newtheorem{question}[theorem]{Question}

\newcommand{\p}{\mathbb{P}}
\newcommand{\C}{\mathbb{C}}
\newcommand{\Q}{\mathbb{Q}}
\newcommand{\A}{\mathbb{A}}

\newcommand{\Z}{\mathbb{Z}}

\newcommand{\F}{\mathbb{F}}
\newcommand{\Hom}{\mathrm{Hom}}

\newcommand{\Spec}{\operatorname{Spec}}
\newcommand{\Km}{\operatorname{Km}}

\newcommand{\Br}{\operatorname{Br}}
\newcommand{\Gal}{\operatorname{Gal}}
\newcommand{\Pic}{\operatorname{Pic}}
\newcommand{\NS}{\operatorname{NS}}
\newcommand{\HH}{\operatorname{H}}
\newcommand{\res}{\operatorname{res}}

\newcommand{\Aut}{\operatorname{Aut}}

\newcommand{\tors}{\operatorname{tors}}

\begin{document}

\title[The Brauer--Manin obstruction and ranks of twists]{The Brauer--Manin obstruction on Kummer varieties and ranks of twists of abelian varieties}
\author{David Holmes and Ren\'e Pannekoek}
\date{\today}
\maketitle
\begin{abstract}
Let $r > 0$ be an integer. We present a sufficient condition for an abelian variety $A$ over a number field $k$ to have infinitely many quadratic twists of rank at least $r$, in terms of the density properties of rational points on the Kummer variety $\Km(A^r)$ of the $r$-fold product of $A$ with itself. 

One consequence of our results is the following. Fix an abelian variety $A$ over $k$, and suppose that for some $r>0$ the Brauer--Manin obstruction to weak approximation on the Kummer variety $\Km(A^r)$ is the only one. Then $A$ has a quadratic twist of rank at least $r$. Hence if the Brauer--Manin obstruction is the only one to weak approximation on all Kummer varieties, then ranks of twists of any positive-dimensional abelian variety are unbounded. 
\end{abstract}

\section{Introduction}


\subsection{Ranks of twists of abelian varieties}

Given an abelian variety $A$ over the number field $k$, the abelian group $A(k)$ is finitely generated by the Mordell--Weil theorem. Hence there exists an isomorphism
$$
A(k) \cong A(k)_{\tors} \oplus \Z^r
$$
for some non-negative integer $r$, which is called the \emph{rank} of $A(k)$, or sometimes simply the rank of $A$ when the ground field $k$ is clear. The rank of an abelian variety is a much-studied invariant, about which many open questions remain. 
\begin{enumerate}
\item For a fixed number field $k$, and an integer $d>0$, is the rank of $A$ bounded as $A$ ranges over all $d$-dimensional abelian varieties over $k$?
\item For a fixed number field $k$, and an abelian variety $A$ over $k$, is the rank of $A^c$ bounded as $A^c$ ranges over all quadratic twists of $A$?
\end{enumerate}
The questions above are wide open. For example, it is not known whether there exists an elliptic curve $E/\Q$ of rank at least 29. Honda \cite{MR0155828} conjectured that the second question has a positive answer if $\dim(A)=1$, but this conjecture is not now uniformly believed \cite{rubin_silverberg}. We will not answer either of these questions, but we will relate them to another open question concerning the Brauer--Manin obstruction to weak approximation on Kummer varieties.

\subsection{Weak approximation} 
Let $X$ be a smooth, projective, geometrically irreducible variety over a number field $k$. We denote by $X(\A_k)$ the topological space of adelic points of $X$. By the properness of $X$, there is a canonical bijection
$$
X(\A_k) \stackrel{\sim}{\rightarrow} \prod_{v \in \Omega_k} X(k_v),
$$
where $\Omega_k$ is the set of places of $k$. One can ask whether the set $X(k)$ of rational points on $X$ is \emph{dense} in $X(\A_k)$, in other words, whether $X$ satisfies \emph{weak approximation}.

For certain special classes of varieties, the answer to this question is known to be affirmative. For example, this is the case if $X$ is equal to $\p^n_k$ for some integer $n \geq 0$, if $X$ is a quadric hypersurface in $\p^n_k$, if $X$ is a cubic hypersurface in $\p^n_k$ and $n \geq 16$, if $X$ is the intersection of two quadrics in $\p^n_k$ and $n \geq 8$, or if $X$ is a del Pezzo surface of degree $\geq 5$ (see \cite{harari} and \cite[Theorem 29.4]{cubic}).  

It is certainly not true that weak approximation holds for general (smooth, projective, geometrically integral) varieties $X$ over $k$. Some counterexamples to weak approximation can be explained by the \emph{Brauer--Manin pairing}. This is a pairing 
$$
\left\langle.,.\right\rangle \colon X(\A_k) \times \Br(X) \rightarrow \Q/\Z,
$$
defined using local class field theory (see \cite[5.2]{torsors}). The pairing $\left\langle .,. \right\rangle$ is a homomorphism whenever the first variable is fixed, and a continuous map whenever the second variable is fixed, if $\Q/\Z$ is given the discrete topology. The \emph{Brauer--Manin set} $X(\A_k)^{\Br}$ of $X$ is defined as the left kernel of the Brauer--Manin pairing. It is a closed subset of $X(\A_k)$. It follows from global class field theory that $X(k)$ is contained in $X(\A_k)^{\Br}$, and hence the same is true for its topological closure, denoted $\ol{X(k)}$ (again, see \cite[5.2]{torsors}). 


If $X(\A_k)^{\Br}$ is strictly contained in $X(\A_k)$, then $X(k)$ is not dense in $X(\A_k)$; in this case, we say that there is a \emph{Brauer--Manin obstruction to weak approximation}. If $\ol{X(k)}$ equals $X(\A_k)^{\Br}$, then we say that the Brauer--Manin obstruction to weak approximation on $X$ \emph{is the only one}. 

In some cases it is known that the Brauer--Manin obstruction to weak approximation on $X$ is the only one; for example, this is the case if $X$ is a del Pezzo surface of degree $4$ such that $X(k) \neq \emptyset$ (see \cite{salsko}), if $k=\Q$ and $X$ admits a conic bundle structure $X \rightarrow \p^1_{\Q}$ with at least one degenerate fibre and such that all degenerate fibres are defined over $\Q$ (see \cite{browningetal}), or if $X$ is a smooth compactification of a homogeneous space of a connected linear group with connected stabilizers (see \cite{borovoi}). On the other hand, examples of $X$ with $\ol{X(k)} \neq X(\A_k)^{\Br}$ can be found in many places (e.g., see \cite{beyond}). 

Lastly, we mention the following conjecture by Jean-Louis Colliot-Th\'el\`ene \cite{avr}: 
\begin{conjecture}\label{jlct}
Let $X$ be a smooth, projective, geometrically integral variety over a number field $k$, and assume that $X_{\ol{k}}$ is birationally equivalent with $\p^n_{\ol{k}}$ for some integer $n$. Then the Brauer--Manin obstruction to weak approximation on $X$ is the only one.
\end{conjecture}

\subsection{Kummer varieties and the Brauer--Manin obstruction}

For an abelian variety $B$ over $k$, we write $\Km(B)$ for its \emph{Kummer variety}, i.e. $\Km(B)$ is the   quotient of $B$ by multiplication by $-1$; it is a geometrically integral projective variety. We note that, if $\dim(B)>1$, then $\Km(B)$ is never smooth. Conjecture \ref{jlct} suggests that the following question could be a reasonable one.
\begin{question}\label{jlct_kummer}
Let $k$ be a number field and $A$ an abelian variety over $k$. Let $r > 0$ an integer. Let $X$ be a smooth projective variety birationally equivalent to the Kummer variety $\Km(A^r)$. Is it
true that the Brauer--Manin obstruction is the only obstruction to weak approximation on X?
\end{question}
In the present paper, we investigate the implications of a positive answer to Question \ref{jlct_kummer} for ranks of twists of abelian varieties over $k$. In this paper, we will prove:

\begin{theorem}[see Corollary \ref{back_to_intro_result}]
\label{intro_result}
Let $k$ be a number field. Let $A$ be a positive-dimensional abelian variety over $k$ and let $r>0$ be an integer. Assume that the Brauer--Manin obstruction to weak approximation on $X$ is the only one, where $X$ is some smooth projective model of $\Km(A^r)$. Then $A$ has infinitely many quadratic twists of rank at least $r$.
\end{theorem}

In fact, the conclusion of the theorem holds under the weaker assumption that there exists a density-zero set $T$ of places of $k$ such that $X(k)$ is dense in $\prod_{v \notin T} X(k_v)$ (see Remark \ref{density_zero_remark} and Theorem \ref{main}). Hence, if the extension of Conjecture \ref{jlct} to smooth projective models of Kummer varieties of the form $\Km(A^r)$ were to hold true, then ranks of quadratic twists of any given positive-dimensional abelian variety are unbounded.




\section{Rational points on a Kummer variety}

In this section, we let $k$ be a field of which the characteristic is different from $2$. For any abelian variety $B$ over $k$, we write $B_0$ for the complement in $B$ of its $2$-torsion subscheme. We further denote by $\Km_0(B)$ the quotient of $B_0$ by its natural $\twogroup$-action. Since the $\twogroup$-action on $B_0$ has no fixed points, it endows $B_0$ with the structure of $\twogroup$-torsor over $\Km_0(B)$. 

Throughout this section, we fix an abelian variety $B$. Let $\wh{X}$ denote $\Km(B)$ and let $q \colon B \rightarrow \wh{X}$ denote the quotient map.  

\subsection{Quadratic twists}
\label{quad_twists}
We first recall some standard facts about quadratic twists of abelian varieties.

Let $L$ be a field extension of $k$ with separable closure $\ol{L}$, and let $\Gamma_L = \Gal(\ol{L}/L)$ be the absolute Galois group of $L$. Let $X$ be a scheme over $L$. Recall that a scheme $X'$ over $L$ is said to be a \emph{twist} of $X$ if there exists an isomorphism $\phi \colon X_{\ol{L}} \stackrel{\sim}{\rightarrow} X'_{\ol{L}}$, where the subscripts denote base-change. If $X'$ is a twist of $X$, and $\phi \colon X_{\ol{L}} \stackrel{\sim}{\rightarrow} X'_{\ol{L}}$ is an isomorphism, then we may associate with the pair $(X',\phi)$ the cocycle $c\colon\Gamma_L \rightarrow \Aut(X_{\ol{L}})$ that maps $\sigma \in \Gamma_L$ to the automorphism ${}^{\sigma}(\phi^{-1}) \circ \phi$. One says that $X'$ is the twist of $X$ \emph{by the cocycle $c$}. This defines an injective map from the set of twists of $X$ up to isomorphism to $\HH^1(L,\Aut(X_{\ol{L}}))$. If $X$ is e.g.~a quasi-projective variety, then this map is a bijection. If this is the case, then there is a natural way to associate with a cocycle $c\colon \Gamma_L\rightarrow\Aut(X_{\ol{L}})$ a pair $(X',\phi)$ as above  (for the last statements, see \cite[III.1.3]{serre_galois}).

Let $c$ be an element of $\HH^1(L,\twogroup)$. Since the $\Gamma_L$-action on $\twogroup$ is trivial, we have $\HH^1(L,\twogroup)=\Hom(\Gamma_L,\twogroup)=L^{\ast}/L^{\ast 2}$, so we may interpret the elements of $\HH^1(L,\twogroup)$ as cocycles $\Gamma_L \rightarrow \twogroup$. Then there exists a scheme $B^c$ over $L$ and an isomorphism $\phi \colon B_{\ol{L}} \stackrel{\sim}{\rightarrow} B^c_{\ol{L}}$ giving rise to the cocycle $c$. Since $B_{\ol{L}}$ is an abelian variety, we can use $\phi$ to endow $B^c_{\ol{L}}$ with the structure of an abelian variety over $\ol{L}$; it is easy to see that this structure descends to $L$. Hence $B^c$ is an abelian variety over $L$ that is the twist of $B_L$ by the cocycle $c$; we call $B^c$ the \emph{quadratic twist} of $B_L$ by the cocycle $c$.

\subsection{Quadratic twists and Kummer varieties}

Again, we let $L$ be any field extension of $k$, with separable closure $\ol{L}$. For all elements $c$ in $\HH^1(L,\twogroup)$, we consider the pair $(B^c,\phi^c)$, where $B^c$ is the quadratic twist of $B_L$ corresponding to $c$, and $\phi^c \colon B_{\ol{L}} \stackrel{\sim}{\rightarrow} B^c_{\ol{L}}$ the corresponding isomorphism. We then define a map 
\begin{equation*}
q^c \colon B^c_{\ol{L}} \rightarrow \wh{X}_{\ol{L}}
\end{equation*}
by setting $q^c = q \circ (\phi^c)^{-1}$, where we have re-used the letter $q$ to denote the base-change to $\ol{L}$ of the quotient map $q \colon B \rightarrow \wh{X}$ defined above. It is easily verified that $q^c$ is defined over $L$; hence we obtain a morphism
\begin{equation*}
q^c \colon B^c \rightarrow \wh{X}_{L}.
\end{equation*}
Moreover, since $\phi^c$ is an isomorphism of abelian varieties, the morphism $q^c$ is again the quotient map for the $\twogroup$-action on $B^c$.

\begin{proposition}\label{commutative}
Let $c \in \HH^1(k,\twogroup)$ and $d \in \HH^1(L,\twogroup)$ be such that the restriction of $c$ to $\HH^1(L,\twogroup)$ equals $d$. Write $B^c_L$ for the base-change of $B^c$ to $L$ and $B^d_L$ for the twist of $B_L$ by the cocycle $d$. Then there exists an isomorphism
$$
F_{c,d} \colon B^c_L \stackrel{\sim}{\rightarrow} B^{d}_L
$$
such that the following diagram is commutative
\[
\xymatrix{
B^c_L \ar[dr]_{q^c} \ar[rr]^{F_{c,d}}_{\cong} & & B^{d}_L \ar[dl]^{q^{d}}  \\\
& \wh{X}_L &  
}
\]
where $\wh{X}_L$ is the base-change of $\wh{X}$ to $L$.
\end{proposition}
\begin{proof}
Let $\phi^c \colon B_{\ol{k}} \stackrel{\sim}{\rightarrow} B^c_{\ol{k}}$ and $\phi^d \colon B_{\ol{L}} \stackrel{\sim}{\rightarrow} B^d_{\ol{L}}$ be the isomorphisms corresponding to the cocycles $c$ and $d$, as chosen above. Let $F_{c,d}$ be the $\ol{L}$-isomorphism
$$
\phi^d \circ (\phi^c)^{-1}\colon B^c_{\ol{L}} \stackrel{\sim}{\rightarrow} B^d_{\ol{L}}.
$$
Since $\res_{L/k}(c)=d$, we have that $F_{c,d}$ is defined over $L$. Finally, we have
$$
q_c = q \circ (\phi_c)^{-1} = q \circ \phi_d^{-1} \circ F_{c,d} = q_d \circ F_{c,d},
$$
hence the diagram is commutative. 
\end{proof}

\subsection{A partition of rational points on a Kummer variety}
We write $X_0$ for $\Km_0(B)$. For each $c$ in $\HH^1(L,\twogroup)$, we write $q^c \colon B^c_0 \rightarrow X_0$ for the restriction of $q^c \colon B^c \rightarrow \wh{X}$. As observed above, the $\twogroup$-action on $B_0$ endows it with the structure of a $\twogroup$-torsor over $X_0$.

\begin{lemma}
\label{fundamental}
Let the notation be as before in this section. The morphisms $q^c \colon  B_0^c \rightarrow X_0$ induce a bijection
\begin{equation}
\label{star}
\coprod_{c \in \HH^1(L,\twogroup)} B_0^c(L)/\twogroup(L) \stackrel{\sim}{\rightarrow} X_0(L).
\end{equation}
Assume that $k$ is a number field and that $L$ is a completion of $k$. Then $\HH^1(L,\twogroup)$ is a finite group. If we endow the sets $B_0^c(L)$ and $X_0(L)$ with the topologies coming from the one on $L$, and the sets $B^c_0(L)/\twogroup(L)$ with the quotient topology, then the bijection \eqref{star} is in fact a homeomorphism. 
\end{lemma}

\begin{proof}
The first part of the lemma is a consequence of eq. (2.12) of \cite{torsors} and the discussion leading up to it. We give a sketch of the idea for the sake of completeness. Let $P \in X_0(L)$ be arbitrary. The fibre $q^{-1}(P)$ over $P$ is a $\twogroup$-torsor over $\Spec(L)$. Such torsors are classified by the Galois cohomology group $\HH^1(L,\twogroup)$; hence $q^{-1}(P)$ determines an element $c' \in \HH^1(L,\twogroup)$. For an arbitrary element $c \in \HH^1(L,\twogroup)$, it is then not hard to see that the fibre $(q^c)^{-1}(P)$ gives a $\twogroup$-torsor corresponding to the element $cc'$ of $\HH^1(L,\twogroup)$. Since a torsor over $\Spec(L)$ has an $L$-point if and only if it corresponds to the identity of $\HH^1(L,\twogroup)$, we see that $P \in q^c(B^c_0(L))$ if and only if $cc'$ is the identity of $\HH^1(L,\twogroup)$.

For the second part, we assume that $k$ is a number field and that $L$ is a completion of $k$. The finiteness of $\HH^1(L,\twogroup)$ is then a well-known fact. Observe that the maps $B^c_0(L) \rightarrow X_0(L)$ induced by the $q^c$ are continuous, since they are induced by morphisms between projective $L$-schemes; by the definition of the quotient topology, the maps $B^c_0(L)/\twogroup(L) \rightarrow X_0(L)$ are likewise continuous. This shows that \eqref{star} is a continuous bijection. To prove our claim it suffices to show that the map \eqref{star} is open; for this, it suffices to show that each of the $B^c_0(L)/\twogroup(L) \rightarrow X_0(L)$ is open. Let $V \subset B^c_0(L)/\twogroup(L)$ be an open subset and let $U$ be its image in $X_0(L)$. The preimage $V'$ of $V$ in $B^c_0(L)$ is open by the continuity of the quotient map $B^c_0(L) \rightarrow B^c_0(L)/\twogroup(L)$, and we have $U = q^c(V')$. Since $q^c$ is an \'etale morphism, it is open by the implicit function theorem for $L$. Therefore $U$ is open, which establishes the claim.
\end{proof}

\begin{remark}\label{product_topologies}
\upshape
Assume that $L=\prod_{i=1}^n L_i$ is a finite product of field extensions $L_i$ of $k$. Then \eqref{star} still holds. Indeed, the set $X_0(L)$ is canonically isomorphic to $\prod_{i=1}^n X_0(L_i)$, and the set $B_0^c(L)/\twogroup(L)$ is canonically isomorphic to the product $\prod_{i=1}^n B_0^c(L_i)/\twogroup(L_i)$. If moreover $k$ is a number field and the $L_i$ are completions of $k$, and we endow the various sets with their natural product topologies, then \eqref{star} is again a homeomorphism.
\end{remark}

\subsection{A useful diagram}
\label{varying}

We continue the notation from the previous parts of this section. In the diagram of Proposition \ref{commutative_diagram}, we will write $\HH^1(k)$ for $\HH^1(k,\twogroup)$ and $\HH^1(L)$ for $\HH^1(L,\twogroup)$. 
\begin{proposition}
\label{commutative_diagram}
There exist maps $f_{L/k}$, $\wt{f}_{L/k}$, and $F_{L/k}$ rendering the following diagram commutative:
\begin{equation*}
\xymatrix{
X_0(k)  \ar[r] & X_0(L)  \\\
  \coprod\limits_{c \in \HH^1(k)} B_0^c(k)/\twogroup(k) \ar@{^(->}[d] \ar[u]_{\cong} \ar[r]^{f_{L/k}} & \coprod\limits_{d \in \HH^1(L)} B_0^d(L)/\twogroup(L) \ar[u]_{\cong} \ar@{^(->}[d]   \\\ 
\coprod\limits_{c \in \HH^1(k)} B^c(k)/\twogroup(k) \ar[r]^{\wt{f}_{L/k}} & \coprod\limits_{d \in \HH^1(L)} B^d(L)/\twogroup(L)   \\\
\coprod\limits_{c \in \HH^1(k)} B^c(k) \ar@{->>}[u] \ar[r]^{F_{L/k}}         & \coprod\limits_{d \in \HH^1(L)} B^d(L)  \ar@{->>}[u]  
}
\end{equation*}
Here, the maps labeled $\cong$ are the ones induced by Lemma \ref{fundamental}, and the unlabeled maps are the natural inclusion and quotient maps.
\end{proposition}

\begin{proof}
It suffices to define $F_{L/k}$, and check that it induces maps $\wt{f}_{L/k}$ and $f_{L/k}$ that make the diagram commutative. Let $F_{L/k}$ be the map
$$
F_{L/k} \colon \coprod_{c \in \HH^1(k,\twogroup)} B^c(k) \rightarrow \coprod_{d \in \HH^1(L,\twogroup)} B^d(L)
$$
that when restricted to $B^c(k)$ equals the map
$$
F_{c,d} \colon B^c(k) \rightarrow B^d(L),
$$
where $F_{c,d}$ is as defined in the proof of Proposition \ref{commutative}, and where $d \in \HH^1(L,\twogroup)$ equals the restriction $\res_{L/k}(c)$ of the cocycle $c$ to $L$. By passage to the quotient, the map $F_{L/k}$ induces a map $\wt{f}_{L/k}$ that makes the lower square of the diagram commutative. Likewise, if we let $f_{L/k}$ be the restriction of $\wt{f}_{L/k}$, it is clear that the middle square is commutative. Finally, it follows from Proposition \ref{commutative} and the definition of $f_{L/k}$ that the top square is commutative.
\end{proof}

\begin{remark}\upshape
It follows from Lemma \ref{fundamental} that if $L = \prod_{i=1}^r L_i$ is a finite product of field extensions of $k$, then Proposition \ref{commutative_diagram} carries over word for word. Now assume additionally that $k$ is a number field, and the $L_i$ are completions of $k$. Then the maps appearing in the second column of the diagram are continuous, where the sets are endowed with their obvious topologies. By Lemma \ref{fundamental} and Remark \ref{product_topologies}, the top-right vertical map is a homeomorphism. Furthermore, since the complement of each $B^c_0(L_i)$ in $B^c(L_i)$ is a finite set, it is easy to see that the downward facing map in the second column is the inclusion of an open subset.
\end{remark}


\section{A criterion for the existence of high-rank twists}

We fix a number field $k$. In this section, we prove a criterion for an abelian variety $A$ to possess a quadratic twist with rank at least $r$, in terms of density properties of rational points on the Kummer variety of the $r$-fold product $A^r$. 


\begin{theorem}
\label{criterion_rank_r_twist}
Let $A$ be an abelian variety over the number field $k$. Let $r>0$ be an integer and denote $B = A^r$. Let $p$ be a prime. Let $S \subset \HH^1(k,\twogroup)$ be a finite subset. Assume $L_1,\ldots,L_r$ are pairwise non-isomorphic non-archimedean completions of $k$, all of good reduction for $A$, such that
\begin{itemize}
\item[(i)] each group $A(L_i)$ has an element $P_i$ of order $p$;
\item[(ii)] for every quadratic twist $A^c$ of $A$ such that either $c \in S$ or $A^c(k)[p] \neq 0$, we have, for each $i$, that $A^c(k)$ is contained in $pA^c(L_i)$;
\item[(iii)] $\Km_0(B)(k)$ is dense in $\prod_{i=1}^r  \Km_0(B)(L_i)$.
\end{itemize}
Then there exists $c \in \HH^1(k,\twogroup)$ such that $c \notin S$ and the rank of $A^c$ is at least $r$.
\end{theorem}

The rest of this section is devoted to the proof of Theorem \ref{criterion_rank_r_twist}. Throughout this section then, let $A$ be an abelian variety over the number field $k$, let $r > 0$ be an integer, let $B = A^r$ be the $r$-fold product of $A$, let $p$ be a prime, and let $L_1,\ldots,L_r$ be pairwise non-isomorphic non-archimedean completions of $k$ satisfying conditions (i)-(iii) from Theorem \ref{criterion_rank_r_twist}. We will write $L$ for the product $\prod_{i=1}^r L_i$.

We need three intermediate lemmas. After Lemma \ref{surj_homs}, some additional definitions will be made that play a crucial role in the proof. 

\begin{lemma}\label{surj_homs}
For each $i$, there exists a surjective group homomorphism $\pi_i \colon A(L_i) \rightarrow \F_p$.
\end{lemma}
\begin{proof}
Indeed, by a well-known result about abelian varieties over non-archimedean local fields \cite{mattuck}, the group $A(L_i)$ has a finite-index subgroup isomorphic to $W = \mathscr{O}_{L_i}^{\dim(A)}$, where $\mathscr{O}_{L_i}$ denotes the ring of integers of $L_i$. Since $W$ is torsion-free, the image of $P_i$ in the finite quotient $A(L_i)/W$ must have order $p$. Clearly then, there exists a surjection $A(L_i)/W \rightarrow \F_p$. This establishes the lemma.
\end{proof}

Let $V=\F_p^r$. Let the group $\twogroup(L) = \prod_{i=1}^r \twogroup(L_i)$ act on $V$ in such a way that the $i$-th basis element $m_i = (1,\ldots,-1,\ldots,1)$ of $\twogroup(L)$, with $-1$ at the $i$-th coordinate, multiplies the $i$-th component of every vector in $V$ by $-1$, while leaving the other components unchanged. We denote by $\pi'=(\pi_1,\ldots,\pi_r)$ the combined map from $A(L) = \prod_{i=1}^r A(L_i)$ to $V$. Note that the map $\pi'$ is $\twogroup(L)$-equivariant for the natural $\twogroup(L)$-action on $A(L)$. Taking into account that $B(L) = A(L)^r$, the map $\pi'$ induces a map 
$$
\pi\colon B(L) \rightarrow V^r.
$$
We endow $V^r$ with the ``diagonal'' $\twogroup(L)$-action: regarding the elements of $V$ as column vectors and the elements of $V^r$ as $r$-by-$r$-matrices, the basis vector $e_i \in \twogroup(L)$ acts on an element of $V^r$ by multiplying its $i$-th row by $-1$. The map $\pi$ is then $\twogroup(L)$-equivariant. Finally, define $\Delta \subset V^r$ as the subset containing the $v \in V^r$ having determinant zero.

\begin{lemma}
\label{the_counting_argument}
The image of $\Delta$ under the quotient map $V^r \rightarrow V^r/\twogroup(L)$ is a proper subset of $V^r/\twogroup(L)$.
\end{lemma}
\begin{proof}
The $\twogroup(L)$-action on $V^r$ sends determinant-zero matrices to determinant-zero matrices. Hence the image of $\Delta$ in $V^r/\twogroup(L)$ consists precisely of the orbits of determinant-zero matrices, which of course form a proper subset of $V^r/\twogroup(L)$.
\end{proof}

\begin{lemma}
\label{linear_dependencies}
Let $c$ be an element of $\HH^1(k,\twogroup)$ such that for all $i \in \{1,2,\ldots,r\}$ we have that $\res_{L_i/k}(c)$ is the identity of $\HH^1(L_i,\twogroup)$. Assume that either $c \in S$ or the rank of $A^c(k)$ is less than $r$. We claim: the image of the composition of group homomorphisms
$$
B^c(k) \rightarrow B^c(L) \rightarrow B(L) \stackrel{\pi}{\rightarrow} V^r,
$$
is contained in $\Delta$. Here, the first arrow is the natural inclusion, and the second arrow is induced by the isomorphisms $F^{(i)}_{c,1} \colon B^c_{L_i} \stackrel{\sim}{\rightarrow} B_{L_i}$ provided by Proposition \ref{commutative}.
\end{lemma}
\begin{proof}
First, suppose that $c$ is such that $c \in S$ or $A^c(k)[p] \neq 0$. Then by assumption (ii) of Theorem \ref{criterion_rank_r_twist}, the image of $B^c(k)$ 
is contained in $pB^c(L)$, and therefore its image in $V^r$ is identically $0$. We may therefore suppose that $c$ is not contained in $S$, so that the rank of $A^c$ is less than $r$, and that $A^c(k)[p]=0$. Let $Q = (Q_1,\ldots,Q_r) \in A^c(k)^r$, and let $\ol{Q}=(\ol{Q}_1,\ldots,\ol{Q}_r)$ denote its image in $V^r$, so that the $\ol{Q}_i$ are elements of $V$ such that $\ol{Q}_i = \pi_i(Q_i)$ for each $i$. Since the rank of $A^c(k)$ is at most $r-1$, we have $a_1 Q_1 + a_2 Q_2 + \ldots + a_r Q_r = 0$ for some integers $a_1,a_2,\ldots,a_r$. Since $A^c(k)[p] = 0$, we may assume that the $a_i$ are not all divisible by $p$. Then we must have a non-trivial relation $a_1 \ol{Q}_1 + a_2 \ol{Q}_2 + \ldots + a_r \ol{Q}_r = 0$, in $V$, showing that $\ol{Q} \in \Delta$.
\end{proof}
\begin{proof}[Proof of Theorem \ref{criterion_rank_r_twist}]
We assume that, for all $c \notin S$, the rank of $A^c$ is less than $r$. Let $\HH^1(k)_1 \subset \HH^1(k,\twogroup)$ be the subset consisting of those $c \in \HH^1(k,\twogroup)$ whose restriction to each of the $L_i$ is the identity. We then consider the following diagram.

\[
\xymatrix{
\coprod\limits_{c \in \HH^1(k)_1} B^c(k)/\twogroup(k) \ar[r]^{~~~\wt{f}} & B(L)/\twogroup(L) \ar@{->>}[r]^{\ol{\pi}} & V^r/\twogroup(L)  \\\
\coprod\limits_{c \in \HH^1(k)_1} B^c(k) \ar@{->>}[u] \ar[r]         & B(L)  \ar@{->>}[u]  \ar@{->>}[r]^{\pi} & V^r \ar@{->>}[u]
}
\]

The maps forming the left-hand square are induced by those from the diagram of Proposition \ref{commutative_diagram}. It was shown there that the left-hand square is commutative. The map $\pi \colon B(L) \rightarrow V^r$ was defined earlier in the proof, and it was observed to be $\twogroup(L)$-equivariant. We define $\ol{\pi} \colon B(L)/\twogroup(L) \rightarrow V^r/\twogroup(L)$ as the map induced by $\pi$, while we let the right-hand vertical map be the quotient map; with these choices, it is clear that the entire diagram is commutative. We consider $V^r$ and $V^r/\twogroup(L)$ as topological spaces endowed with the discrete topology. Then $\pi$ is continuous by the fact that the topology on $B(L)$ is the profinite one, while $\ol{\pi}$ is continuous by the quotient property of $B(L)/\twogroup(L)$.

By Lemma \ref{linear_dependencies}
, the composite map
\begin{equation}
\label{non_surjective_quotient_map}
\coprod_{c \in \HH^1(k)_1} B^c(k)/\twogroup(k) \rightarrow V^r/\twogroup(L)
\end{equation}
factors via $\Delta/\twogroup(L)$. By Lemma \ref{the_counting_argument}, the image of the map \eqref{non_surjective_quotient_map} is a proper subset of $V^r/\twogroup(L)$. Since $\ol{\pi}$ is a continuous surjection and $V^r/\twogroup(L)$ is discrete, this means that the image of $\wt{f}$ does not lie dense. But then the image of the map $f_{L/k}$ (as in the diagram from Proposition \ref{commutative_diagram}) does not lie dense. Again by the diagram of Proposition \ref{commutative_diagram}, this shows that $X_0(k)$ cannot be dense in $X_0(L)$, which is a contradiction.
\end{proof}

\section{Brauer groups of Kummer varieties over number fields}

In this section, we let $k$ be a number field. Let $B$ be an abelian variety over $k$. Let $\wh{X} = \Km(B)$ be its Kummer variety and let $X$ be any smooth projective model of $\wh{X}$. We denote by $\Br(X)$ the Brauer group of $X$, and by $\Br_0(X)$ the image in $\Br(X)$ of $\Br(k)$.
\begin{proposition}
\label{brfinite}
The quotient $\Br(X)/\Br_0(X)$ is finite.
\end{proposition}
\begin{proof}
Let $\Br_1(X)=\ker\left( \Br(X) \rightarrow \Br(X_{\ol{k}})\right)$ denote the subgroup of algebraic elements in $\Br(X)$. (As usual, we write ${X}_{\ol{k}}$ for the base-change of $X$ to $\ol{k}$.) It suffices to prove that the quotients $\Br_1(X)/\Br_0(X)$ and $\Br(X)/\Br_1(X)$ are finite. 

By the Hochschild--Serre spectral sequence, the quotient $\Br_1(X)/\Br_0(X)$ is isomorphic to the Galois cohomology group $\HH^1(k,\Pic({X}_{\ol{k}}))$. Since the topological fundamental group $\pi_1(X(\C))$ vanishes \cite{pi1_of_Kummer}, the variety $X$ has trivial Albanese variety, so that $\HH^1(k,\Pic({X}_{\ol{k}})) = \HH^1(k,\NS({X}_{\ol{k}}))$. Again since $\pi_1(X(\C))$ vanishes, the N\'eron--Severi group of ${X}_{\ol{k}}$ is torsion-free, which implies that $\NS({X}_{\ol{k}})$ is a finitely generated torsion-free abelian group. This implies that $\HH^1(k,\NS({X}_{\ol{k}}))$, and therefore $\HH^1(k,\Pic({X}_{\ol{k}}))$, are both finite.

Write $\Gamma = \Gal(\ol{k}/k)$. The group $\Br(X)/\Br_1(X)$ embeds into $\Br({X}_{\ol{k}})^{\Gamma}$. Let $\res$ be the restriction map from $\Br({X}_{\ol{k}})^{\Gamma}$ to $\Br({B}_{\ol{k}})^{\Gamma}$, where ${B}_{\ol{k}}$ is the base-change of $B$ to $\ol{k}$. The group $\Br({B}_{\ol{k}})^{\Gamma}$ is finite by \cite[Theorem 1.1(i)]{finiteness}; hence the map $\res$ has finite image. The elements of the kernel of $\res$ have order at most $2$ by a standard restriction-corestriction argument. The group $\Br({X}_{\ol{k}})$ is isomorphic to the direct sum of $(\Q/\Z)^{b_2-\rho}$ and a finite abelian group \cite[eq. (7)]{descent_on_Br}, where $b_2$ is the second Betti number of ${X}_{\ol{k}}$ and $\rho$ is the Picard number of ${X}_{\ol{k}}$; therefore $\Br({X}_{\ol{k}})$ has only finitely many elements of order at most $2$. It follows that the kernel of $\res$ is finite. Since its image was already shown to be finite, we must have that $\Br({X}_{\ol{k}})^{\Gamma}$ is finite. Since $\Br(X)/\Br_1(X)$ injects into $\Br({X}_{\ol{k}})^{\Gamma}$, it too is finite, and so we are done.
\end{proof}

\begin{proposition}
\label{brauermaninmain}
If $X$ is a smooth projective model of the Kummer variety of an abelian variety $B$ over $k$, then $X(\A_k)^{\Br}$ is a non-empty open and closed subset of $X(\A_k)$.
\end{proposition}
\begin{proof}
Observe that $B(k)$ is non-empty since it contains $0$. The Lang--Nishimura Lemma applied to the rational map $B \dashrightarrow X$ between smooth and proper varieties then implies that $X(k)$ is non-empty. The subset $X(\A_k)^{\Br}$ is therefore also non-empty, for it contains $X(k)$. It is open and closed, since it is the intersection of the open and closed subsets
$$
X(\A_k)^{\mathcal{A}} = \left\{ (x_v)_v \in X(\A_k) : \left\langle (x_v)_v,\mathcal{A}\right\rangle=0  \right\},
$$
where $\left\langle.,.\right\rangle$ is the Brauer--Manin pairing, and where $\mathcal{A}$ runs through a set of coset representatives of the group quotient $\Br(X)/\Br_0(X)$, which is finite by Proposition \ref{brfinite}.
\end{proof}

\section{The Brauer--Manin obstruction and ranks of twists}

We fix again a number field $k$. 

\begin{lemma}\label{cebotarev}
Let $A$ be an abelian variety over $k$, and let $T$ be any finite set of places of the number field $k$. Let $S' \subset \HH^1(k,\twogroup)$ be a finite subset. For every prime $p$ and every integer $r>0$, there exist pairwise non-isomorphic non-archimedean completions $L_1,\ldots,L_r$ of $k$, none of which arise from any place in $T$, such that for each $i \in \{1,2,\ldots,r\}$ we have:
\begin{itemize}
\item[(i)] the group $A(L_i)$ contains a point of order $p$;
\item[(ii)] for all $c \in S'$ we have $A^c(k) \subset pA^c(L_i)$.
\end{itemize}
\end{lemma}
\begin{proof}
It follows from the finiteness of $S'$ and from the Mordell--Weil theorem that there exists a number field $\ell \supset k$ such that $A(\ell)[p] \neq 0$ and such that for all $c \in S'$ we have $A^c(k) \subset pA^c(\ell)$. The {\v C}ebotarev density theorem implies that the set of places of $k$ splitting completely in $\ell$ has positive density, and the lemma follows.
%
\end{proof}

\begin{remark}\label{density_zero_remark}\upshape
The proof of Lemma \ref{cebotarev} shows that we could have relaxed the assumption that $T$ is finite to $T$ being merely of density zero. We will not make use of this fact, however. 
\end{remark}

\begin{theorem}
\label{main}
Let $A$ be an abelian variety over the number field $k$, and let $r>0$ be an integer. Write $B=A^r$. Let $X$ be a smooth projective variety birationally equivalent to $\Km(B)$. Suppose there exists a finite set $T \subset \Omega_k$ such that, for all $r$-element subsets $T'$ of $\Omega_k \setminus T$, we have that $X(k)$ is dense in $\prod_{v \in T'} X(k_v)$, when the latter set is equipped with the product topology. Then $A$ has infinitely many twists of rank at least $r$.
\end{theorem}

\begin{proof}
Let $S \subset \HH^1(k,\twogroup)$ be any finite subset. It suffices to show that there exists $c \notin S$ such that the rank of $A^c$ is at least $r$. 

By the main result of \cite{silvheights}, the number of $c \in \HH^1(k,\twogroup)$ with $A^c(k)_{\operatorname{tors}} \neq A^c(k)[2]$ is finite. Fix a prime $p>2$. Then Lemma \ref{cebotarev} yields the existence of pairwise non-isomorphic non-archimedean completions $L_1,\ldots,L_r$ of $k$, none of which arise from $T$ and all of which are of good reduction for $A$, such that for $i \in \{1,2,\ldots,r\}$ we have:

\begin{itemize}
\item[(i)] the group $A(L_i)$ contains a point $P_i$ of order $p$;
\item[(ii)] for all $c$ such that either $c \in S$ or $A^c(k)[p] \neq 0$, we have $A^c(k) \subset pA^c(L_i)$.
\end{itemize}

We write $L=\prod_{i=1}^r L_i$. By the assumptions of the theorem, the set $X(k)$ is dense in the topological space $X(L)$. By \cite[Proposition 3.7]{renes_thesis}, the density of $X(k)$ in $X(L)$ implies the density of $\Km_0(B)(k)$ in $\Km_0(B)(L)$. Therefore the conditions of Theorem \ref{criterion_rank_r_twist} are satisfied. By Theorem \ref{criterion_rank_r_twist}, there exists $c \notin S$ such that the rank of $A^c$ is at least $r$, which proves the theorem.
\end{proof}

\begin{corollary}
\label{back_to_intro_result}
Let $k$ be a number field. Let $A$ be a positive-dimensional abelian variety over $k$ and let $r>0$ be an integer. Assume that the Brauer--Manin obstruction to weak approximation on $X$ is the only one, where $X$ is some smooth projective model of $\Km(A^r)$. Then $A$ has infinitely many quadratic twists of rank at least $r$.
\end{corollary}

\begin{proof}
Indeed, suppose that the hypothesis of the corollary is satisfied. Then by Proposition \ref{brauermaninmain} we have that for all abelian varieties $A$ over $k$, if $X$ is a smooth projective variety birationally equivalent to $\Km(A^r)$ for some integer $r>0$, then $X(k)$ is dense in $\prod_{v \notin T} X(k_v)$ for some finite subset $T$ of $\Omega_k$. The result then follows from Theorem \ref{main}.
\end{proof}

\section{Acknowledgements}

We thank Alex Bartel, Martin Bright, Bas Edixhoven, Samir Siksek, Damiano Testa, Tony V\'arilly-Alvarado, and Bianca Viray for discussions on this work, and Ronald van Luijk for valuable suggestions. Finally, we wish to thank the referee for a careful reading of this paper and several helpful suggestions.

\bibliographystyle{plain}
\bibliography{brauermanin}

\begin{thebibliography}{10}

\bibitem{borovoi}
M.~Borovoi.
\newblock The {B}rauer--{M}anin obstruction for homogeneous spaces with
  connected or abelian stabilizer.
\newblock {\em J. Reine Angew. Math.}, 473:181--194, 1996.

\bibitem{browningetal}
T.D. Browning, L.~Matthiesen, and A.N. Skorobogatov.
\newblock Rational points on pencils of conics and quadrics with many
  degenerate fibres.
\newblock Preprint, 2012.
\newblock \texttt{http://arxiv.org/abs/1209.0207}.

\bibitem{avr}
J.L. Colliot-Th\'el\`ene.
\newblock L'arithm\'etique des vari\'et\'es rationnelles.
\newblock {\em Ann. Fac. Sci. Toul.}, 54(2):375--492, 1992.

\bibitem{descent_on_Br}
J.L. Colliot-Th\'el\`ene and A.N. Skorobogatov.
\newblock Descente galoisienne sur le groupe de {B}rauer.
\newblock {\em J. reine angew. Math.}, 682:141--165, 2013.

\bibitem{harari}
D.~Harari.
\newblock Weak approximation on algebraic varieties.
\newblock In {\em Arithmetic of {h}igher-dimensional varieties}, pages 43--60.
  Birkh{\"a}user, 2004.

\bibitem{MR0155828}
Taira Honda.
\newblock Isogenies, rational points and section points of group varieties.
\newblock {\em Japan. J. Math.}, 30:84--101, 1960.

\bibitem{cubic}
Yu.I. Manin.
\newblock {\em Cubic {F}orms: {A}lgebra, {G}eometry, {A}rithmetic}.
\newblock North-Holland Publishing Co., Amsterdam, 2nd edition, 1986.

\bibitem{mattuck}
A.~Mattuck.
\newblock Abelian varieties over {$p$}-adic ground fields.
\newblock {\em Ann. Math.}, 62(1):92--119, 1956.

\bibitem{renes_thesis}
R.~Pannekoek.
\newblock {\em Topological aspects of rational points on {K3} surfaces}.
\newblock PhD thesis, Leiden University, 2013.
\newblock {\tt http://hdl.handle.net/1887/21743}.

\bibitem{rubin_silverberg}
K.~Rubin and A.~Silverberg.
\newblock Ranks of elliptic curves.
\newblock {\em Bull. Amer. Math. Soc.}, 39(4):455--474, 2002.

\bibitem{salsko}
P.~Salberger and A.N. Skorobogatov.
\newblock Weak approximation for surfaces defined by two quadratic forms.
\newblock {\em Duke Math. J.}, 63(2):517--536, 1991.

\bibitem{serre_galois}
J.P. Serre.
\newblock {\em Galois cohomology}.
\newblock Springer, Berlin Heidelberg, 1997.

\bibitem{silvheights}
J.H. Silverman.
\newblock Lower bounds for height functions.
\newblock {\em Duke Math. J.}, 51:395--403, 1984.

\bibitem{beyond}
A.N. Skorobogatov.
\newblock Beyond the {M}anin obstruction.
\newblock {\em Inv. Math.}, 135:399--424, 1999.

\bibitem{torsors}
A.N. Skorobogatov.
\newblock {\em Torsors and rational points}.
\newblock Cambridge {U}niversity {P}ress, Cambridge, 2001.

\bibitem{finiteness}
A.N. Skorobogatov and Yu.G. Zarhin.
\newblock A finiteness theorem for the {B}rauer group of abelian varieties and
  {K3} surfaces.
\newblock {\em J. Alg. Geom.}, 17:481--502, 2008.

\bibitem{pi1_of_Kummer}
E.~Spanier.
\newblock The homology of {K}ummer manifolds.
\newblock {\em Proc. Amer. Math. Soc.}, 7:155--160, 1956.

\end{thebibliography}

\end{document}